\documentclass[a4paper,12pt,final]{amsart}
\usepackage{times,a4wide,mathrsfs,amssymb,amsmath,amsthm,wasysym}

\newcommand{\C}{\mathbb{C}}

\newcommand{\QQ}{\mathbb{Q}}

\newcommand{\PP}{\mathbb{P}}

\newcommand{\oo}{\mathfrak o}

\newcommand{\EE}{\mathcal E}

\newcommand{\gr}{\hbox{Gr}}

\newcommand{\rom}{\romannumeral}

 \makeatletter
\newcommand*{\da@rightarrow}{\mathchar"0\hexnumber@\symAMSa 4B }
\newcommand*{\da@leftarrow}{\mathchar"0\hexnumber@\symAMSa 4C }
\newcommand*{\xdashrightarrow}[2][]{%
  \mathrel{%
    \mathpalette{\da@xarrow{#1}{#2}{}\da@rightarrow{\,}{}}{}%
  }%
}
\newcommand{\xdashleftarrow}[2][]{%
  \mathrel{%
    \mathpalette{\da@xarrow{#1}{#2}\da@leftarrow{}{}{\,}}{}%
  }%
}
\newcommand*{\da@xarrow}[7]{%
  \sbox0{$\ifx#7\scriptstyle\scriptscriptstyle\else\scriptstyle\fi#5#1#6\m@th$}%
  \sbox2{$\ifx#7\scriptstyle\scriptscriptstyle\else\scriptstyle\fi#5#2#6\m@th$}%
  \sbox4{$#7\dabar@\m@th$}%
  \dimen@=\wd0 %
  \ifdim\wd2 >\dimen@
    \dimen@=\wd2 %
  \fi
  \count@=2 %
  \def\da@bars{\dabar@\dabar@}%
  \@whiledim\count@\wd4<\dimen@\do{%
    \advance\count@\@ne
    \expandafter\def\expandafter\da@bars\expandafter{%
      \da@bars
      \dabar@ 
    }%
  }%
  \mathrel{#3}%
  \mathrel{%
    \mathop{\da@bars}\limits
    \ifx\\#1\\%
    \else
      _{\copy0}%
    \fi
    \ifx\\#2\\%
    \else
      ^{\copy2}%
    \fi
  }%
  \mathrel{#4}%
}
\makeatother

\newtheorem{theorem}{Theorem}[section]

\newtheorem{lemma}[theorem]{Lemma}

\newtheorem{corollary}[theorem]{Corollary}
\newtheorem{proposition}[theorem]{Proposition}
\newtheorem{conjecture}[theorem]{Conjecture}
\newtheorem{remark}[theorem]{Remark}
\newtheorem{definition}[theorem]{Definition}
\newtheorem{convention}{Conventions}

\newtheorem{nonumberingc}{Corollary}

\newtheorem{nonumberingp}{Proposition}

\newtheorem{nonumberingt}{Acknowledgements}

\begin{document}
\author[Robert Laterveer]
{Robert Laterveer}

\address{Institut de Recherche Math\'ematique Avanc\'ee,
CNRS -- Universit\'e 
de Strasbourg,\
7 Rue Ren\'e Des\-car\-tes, 67084 Strasbourg CEDEX,
FRANCE.}
\email{robert.laterveer@math.unistra.fr}

\title{A remark on the Chow ring of some hyperk\"ahler fourfolds}

\begin{abstract} Let $X$ be a hyperk\"ahler variety. Voisin has conjectured that the classes of Lagrangian constant cycle subvarieties in the Chow ring of $X$ should lie in a subring injecting into cohomology.
We study this conjecture for the Fano variety of lines on a very general cubic fourfold. 
 \end{abstract}


\maketitle

\section{Introduction}

For a smooth projective variety $X$ over $\C$, let $A^i(X):=CH^i(X)_{\QQ}$ denote the Chow groups (i.e. the groups of codimension $i$ algebraic cycles on $X$ with $\QQ$--coefficients, modulo rational equivalence). 
Let $A^i_{hom}(X)$ denote the subgroup of homologically trivial cycles. 

As is well--known, the world of Chow groups is still largely shrouded in mystery, its map containing vast unexplored regions only vaguely sketched in by conjectures \cite{B}, \cite{J2}, \cite{J4}, \cite{Kim}, \cite{Mur}, \cite{Vo}, \cite{MNP}. One region on this map that holds particular interest is that of hyperk\"ahler varieties (i.e. projective irreducible holomorphic symplectic manifolds \cite{Beau1}, \cite{Beau0}). 
Here, motivated by results for $K3$ surfaces and for abelian varieties, in recent years significant progress has been made in the understanding of Chow groups \cite{BV}, \cite{V13}, \cite{V14}, \cite{V17}, \cite{SV}, \cite{V6}, \cite{Rie}, \cite{Rie2}, \cite{LFu2}, \cite{Lin}, \cite{Lin2}, \cite{FTV}.

It is expected that for a hyperk\"ahler variety $X$, the Chow groups split in a finite number of pieces
    \[ A^i(X) = \bigoplus_j A^i_{(j)}(X)\ ,\]
  such that $A^\ast_{(\ast)}(X)$ is a bigraded ring and $A^\ast_{(0)}(X)$ injects into cohomology. This was first conjectured by Beauville \cite{Beau3}, who conjectured more precisely that the piece $A^i_{(j)}(X)$ should be isomorphic to the graded
  $\gr^j_F A^i(X)$ for the conjectural Bloch--Beilinson filtration.
    
  What kind of cycles are contained in the subring $A^\ast_{(0)}(X)$ ? Certainly divisors and the Chern classes of $X$ should be in this subring. In addition to this, Voisin has stated the following conjecture:

 \begin{conjecture}[Voisin \cite{V14}]\label{conj2} Let $X$ be a hyperk\"ahler variety of dimension $2m$. 
 
 \noindent
 (\rom1)
 Let $Y\subset X$ be a Lagrangian constant cycle subvariety (i.e., $\dim Y=m$ and the pushforward map $A_0(Y)\to A_0(X)$ has image of dimension $1$). Then
   \[   Y\ \ \ \in\ A^m_{(0)}(X)\ .\]
   
  \noindent
  (\rom2) 
   The subring of $A^\ast(X)$ containing divisors, Chern classes and Lagrangian constant cycle subvarieties injects into cohomology.
    \end{conjecture}
    
  (NB: part (\rom2) follows from part (\rom1), provided the bigrading $A^\ast_{(\ast)}(X)$ has the desirable property that $A^\ast_{(0)}(X)\cap A^\ast_{hom}(X)=0$, which is expected from the Bloch--Beilinson conjectures.)  
 
Evidence for conjecture \ref{conj2} is presented in \cite{V14}. The modest aim of this note is to determine how far conjecture \ref{conj2} can be solved unconditionally in the special case where $X$ is the Fano variety of lines on a cubic fourfold. Here, the {\em Fourier decomposition\/} of Shen--Vial \cite{SV} provides an unconditional splitting $A^\ast_{(\ast)}(X)$ of the Chow ring.
The main result is as follows:

\begin{nonumberingp}[=proposition \ref{cc}] Let $Z\subset\PP^5(\C)$ be a very general smooth cubic fourfold, and let $X$ be the Fano variety of lines in $Z$. Assume
 $Y\subset X$ is a Lagrangian constant cycle subvariety. Then 
   \[ Y\in A^2_{(0)}(X)\ \]
   (where $A^\ast_{(\ast)}(X)$ denotes the Fourier decomposition of \cite{SV}).
  \end{nonumberingp} 
  
  This doesn't settle conjecture \ref{conj2}(\rom2) (because it is not known whether $A^2_{(0)}(X)\cap A^2_{hom}(X)=0$). However, this at least implies some statements along the lines of conjecture \ref{conj2}(\rom2):
  
 \begin{nonumberingc}[=corollaries \ref{cc1} and \ref{cc2}] Let $Z\subset\PP^5(\C)$ be a very general smooth cubic fourfold, and let $X$ be the Fano variety of lines in $Z$. 
 
 \noindent
 (\rom1)
 Let $a\in A^3(X)$ be a $1$--cycle of the form
   \[ a = \displaystyle\sum_{i=1}^r  Y_i\cdot D_i\ \ \ \in A^3(X)\ ,\]
   where $Y_i$ is a Lagrangian constant cycle subvariety and $D_i\in A^1(X)$. Then $a$ is rationally trivial if and only if $a$ is homologically trivial.  
   
  \noindent
  (\rom2) Let $a\in A^4(X)$ be a $0$--cycle of the form
   \[ a = \displaystyle\sum_{i=1}^r  Y_i\cdot b_i\ \ \ \in A^4(X)\ ,\]
   where $Y_i$ is a Lagrangian constant cycle subvariety and $b_i\in A^2(X)$. Then $a$ is rationally trivial if and only if $a$ is homologically trivial.
  \end{nonumberingc}

 \vskip0.6cm

\begin{convention} In this article, the word {\sl variety\/} will refer to a reduced irreducible scheme of finite type over $\C$. A {\sl subvariety\/} is a (possibly reducible) reduced subscheme which is equidimensional. 

{\bf All Chow groups will be with rational coefficients}: we will denote by $A_j(X)$ the Chow group of $j$--dimensional cycles on $X$ with $\QQ$--coefficients; for $X$ smooth of dimension $n$ the notations $A_j(X)$ and $A^{n-j}(X)$ are used interchangeably. 

The notations $A^j_{hom}(X)$, $A^j_{AJ}(X)$ will be used to indicate the subgroups of homologically trivial, resp. Abel--Jacobi trivial cycles.


We use $H^j(X)$ 
to indicate singular cohomology $H^j(X,\QQ)$.
\end{convention}

\section{Preliminaries}

 \subsection{The Fourier decomposition}   
 
 \begin{theorem}[Shen--Vial \cite{SV}]\label{svfour} Let $Z\subset\PP^5(\C)$ be a smooth cubic fourfold, and let $X$ be the Fano variety of lines in $Z$. There is a decomposition
   \[ A^i(X) =\bigoplus_{\stackrel{0\le j\le i}{j\ {\scriptstyle even}}} A^i_{(j)}(X)\ ,\]
   with the following properties:
   
   \noindent
   (\rom1)  $A^i_{(j)}(X) =(\Pi_{2i-j}^X)_\ast A^j(X)$, where $\{\Pi^X_\ast\}$ is a certain self--dual Chow--K\"unneth decomposition;
   
   \noindent
   (\rom2) $A^i_{(j)}(X)\subset A^i_{hom}(X)$ for $j>0$;
   
   \noindent
   (\rom3) if $Z$ is very general, $A^\ast_{(\ast)}(X)$ is a bigraded ring.
   \end{theorem}

 \begin{proof} The decomposition is defined in terms of a Fourier transform, involving the cycle $L\in A^2(X\times X)$ representing the Beauville--Bogomolov class (cf. \cite[Theorem 2]{SV}).
 Points (\rom1) and (\rom2) follow from \cite[Theorem 3.3]{SV}. Point (\rom3) is \cite[Theorem 3]{SV}.
 \end{proof}

 \subsection{Multiplicative structure}
 
 \begin{theorem}[Shen--Vial \cite{SV}]\label{mult} Let $Z\subset\PP^5(\C)$ be a smooth cubic fourfold, and let $X$ be the Fano variety of lines in $Z$.
    There is a distinguished class $l\in A^2_{(0)}(X)$ such that intersection induces an isomorphism
    \[ \cdot l\colon\ \ \ A^2_{(2)}(X)\ \xrightarrow{\cong}\ A^4_{(2)}(X)\ .\]
The inverse isomorphism is given by 
  \[ {1\over 25} L_\ast\colon\ \ \ A^4_{(2)}(X)\ \xrightarrow{\cong}\ A^2_{(2)}(X)\ ,\]
  where $L\in A^2(X\times X)$ is the class defined in \cite[Equation (107)]{SV}.
\end{theorem}    
    
 \begin{proof} This follows from \cite[Theorems 2.2 and 2.4]{SV}.
  \end{proof}

 \subsection{The class $c$}
 \label{scc}
 
 \begin{lemma}[Voisin \cite{V17}, Shen--Vial \cite{SV}]\label{lem} Let $Z\subset\PP^5(\C)$ be a smooth cubic fourfold, and let $X$ be the Fano variety of lines in $Z$. Let $c:=c_2(\EE_2)\in A^2(X)$, where $\EE_2$ is the restriction to $X$ of the tautological rank $2$ vector bundle on the Grassmannian of lines in $\PP^5(\C)$. There exists a constant cycle surface $Y_0\subset X$ such that
   \[ Y_0=c\ \ \ \hbox{in}\ A^2(X)\ .\]
  (In particular, $\cdot c\colon A^2_{hom}(X)\to A^4(X)$ is the zero--map.)
  
  Moreover, if $Z$ is very general then the class $c$ is in $A^2_{(0)}(X)$ (where $A^\ast_{(\ast)}(X)$ is the Fourier decomposition of \cite{SV}). 
 \end{lemma}
 
 \begin{proof} This is well--known. As explained in \cite[Lemma 3.2]{V17}, the idea is to consider $Y\subset X$ defined as the Fano surface of lines contained in $Z\cap H$, where $H$ is a hyperplane in $\PP^5$. For general $H$, the surface $Y$ is a smooth surface of general type which is a Lagrangian subvariety of class $c$ in $A^2(X)$. However, if one takes $H$ such that $Z\cap H$ acquires $5$ nodes, then one obtains a singular surface $Y_0$ which is rational, hence
 $A_0(Y_0)=\QQ$. It follows that $Y_0\subset X$ is a constant cycle subvariety of class $c$ in $A^2(X)$.
 
 The last statement is \cite[Theorem 21.9(\rom3)]{SV}.
 \end{proof}

 \subsection{A result in cohomology}
 
 \begin{definition}[Voisin \cite{V14}] Let $X$ be a hyperk\"ahler variety of dimension $2m$. A Hodge class $a\in H^{2m}(X)\cap F^m$ is {\em coisotropic\/} if
 \[ \cup a\colon\ \ \ H^{2,0}(X)\ \to\ H^{m+2,m}(X) \]
 is the zero--map.
 \end{definition}
 
 (This is \cite[Definition 1.5]{V14}, where coisotropic cohomology classes are defined in any degree $2i$.)

 \begin{proposition}\label{katia} Let $Z\subset\PP^5(\C)$ be a very general smooth cubic fourfold, and let $X$ be the Fano variety of lines in $Z$. Assume
 $a\in H^4(X)$ is coisotropic. Then 
   \[ a=\lambda\cdot c\ \ \ \hbox{in}\ H^4(X)\ ,\]
   where $\lambda\in\QQ$ and $c\in A^2(X)$ is as in lemma \ref{lem}.
   \end{proposition}
   
   \begin{proof} For very general $Z$, it is known that $N^2 H^4(X)$ (which is the subspace of Hodge classes, as the Hodge conjecture is known for $X$) has dimension $2$. This is all that we need for the proof.
   
   For any ample class $g\in A^1(X)$, the $\QQ$--vector space $N^2 H^4(X)$ is generated by $g^2$ and $c$. (These two elements cannot be proportional, as cupping with $g^2$ induces an isomorphism $H^{2,0}(X)\cong H^{4,2}(X)$ by hard Lefschetz, whereas cupping with $c$ is the zero--map $H^{2,0}(X)\to H^{4,2}(X)$.)
   Let us write
   \[ a= \lambda_1 c +\lambda_2 g^2\ \ \ \hbox{in}\ N^2 H^4(X)\ .\]
   The coisotropic condition forces $\lambda_2$ to be $0$, and we are done.
     \end{proof}
     
    \begin{remark} In particular, proposition \ref{katia} implies that any Lagrangian subvariety $Y\subset X$ is proportional to $c$ in cohomology:
      \[ Y=\lambda\cdot c\ \ \ \hbox{in}\ H^4(X)\ .\]
      This was first observed by Amerik \cite[Remark 9]{Am}.
      \end{remark}

 \section{Main result}
 
 \begin{proposition}\label{cc} Let $Z\subset\PP^5(\C)$ be a very general smooth cubic fourfold, and let $X$ be the Fano variety of lines in $Z$. Assume
 $Y\subset X$ is a constant cycle subvariety of codimension $2$. Then 
   \[ Y\in A^2_{(0)}(X)\ .\]
   \end{proposition}
   
  \begin{proof} We assume there is a decomposition
    \[ Y =b_0+b_2\ \ \ \hbox{in}\ A^2_{(0)}(X)\oplus A^2_{(2)}(X)\ ,\]
    with $b_i\in A^2_{(i)}(X)$. We will show that $b_2$ must be $0$.
    
 First, we claim that
    \begin{equation}\label{in0} Y\cdot a\ \ \ \in A^4_{(0)}(X)\ \ \ \forall a\in A^2(X)\ .\end{equation}
    Indeed, the subvectorspace $Y\cdot A^2(X)\subset A^4(X)$ has dimension $1$, as $Y\subset X$ is a constant cycle subvariety. To prove (\ref{in0}), it remains to exclude the possibility that
     \[ \bigl( Y\cdot A^2(X)\bigr) \cap A^4_{(0)}(X) =0\ .\]
   But we know (proposition \ref{katia}) that
   \[ Y=\lambda c\ \ \ \hbox{in}\ H^4(X)\ ,\]
   for some $\lambda\in\QQ^\ast$. Since $c\in A^2_{(0)}(X)$, this implies there is a further decomposition
    \[ Y= \lambda c + b_0^\prime + b_2\ \ \ \hbox{in}\ A^2(X)  \ ,\]
    with $b_0^\prime\in A^2_{(0)}(X)\cap A^2_{hom}(X)$ (which is conjecturally, but not provably, zero).
    Consider the intersection
    \[ Y\cdot c = \lambda c^2 + b_0^\prime\cdot c + b_2\cdot c =\lambda c^2\ \ \ \hbox{in}\ A^4(X)\ .\]
    (Here we have used that $c\cdot A^2_{hom}(X)=0$ in $A^4(X)$, which is lemma \ref{lem} or \cite[Lemma A.3(\rom3)]{SV}.)
    Since $c^2=27\oo_X$ where $\oo_X$ is a certain distinguished generator of $A^4_{(0)}(X)$ \cite[Lemma A.3(\rom1)]{SV}, the intersection $Y\cdot c$ defines a non--zero element in $A^4_{(0)}(X)$. This proves the claim.
    
   To prove the proposition, consider the intersection
    \[ Y\cdot \ell = b_0\cdot\ell + b_2\cdot\ell\ \ \ \hbox{in}\ A^4(X)\ ,\]
    where $\ell$ is the class of theorem \ref{mult}.
    Since $\ell\in A^2_{(0)}(X)$ and $A^\ast_{(\ast)}(X)$ is a bigraded ring, we have that
     $b_i\cdot\ell\in A^4_{(i)}(X)$. It follows from (\ref{in0}) that $Y\cdot\ell\in A^4_{(0)}(X)$ and so 
     \[ b_2\cdot\ell=0\ \ \ \hbox{in}\ A^4_{(2)}(X)\ .\]
     But then, applying theorem \ref{mult}, we find that $b_2=0$ and we are done. 
      \end{proof}

\begin{remark} Let $X$ be the Fano variety of a very general cubic fourfold. We have seen (proposition \ref{katia}) that any Lagrangian constant cycle subvariety $Y$ is proportional to the class $c$ in cohomology. 
 Proposition \ref{cc} suggests that the same should be true modulo rational equivalence: indeed, $Y$ is proportional to $c$ in $A^2(X)$ modulo the 
 ``troublesome part'' $A^2_{(0)}(X)\cap A^2_{hom}(X)$ (which is conjecturally zero).
   \end{remark}

 
 \section{Corollaries}

We present three corollaries that provide weak versions of conjecture \ref{conj2}(\rom2).
  
  \begin{corollary}\label{cc2} Let $Z\subset\PP^5(\C)$ be a very general smooth cubic fourfold, and let $X$ be the Fano variety of lines in $Z$. Let $a\in A^4(X)$ be a $0$--cycle of the form
   \[ a = \displaystyle\sum_{i=1}^r  Y_i\cdot b_i\ \ \ \in A^4(X)\ ,\]
   where $Y_i$ is a Lagrangian constant cycle subvariety and $b_i\in A^2(X)$. Then $a$ is rationally trivial if and only if $a$ is homologically trivial.
   \end{corollary}  
  
   \begin{proof} We know from claim (\ref{in0}) that $a$ is in $A^4_{(0)}(X)$. But $A^4_{(0)}(X)\cong\QQ$ injects into cohomology.
   \end{proof}
   
 \begin{corollary}\label{cc1} Let $Z\subset\PP^5(\C)$ be a very general smooth cubic fourfold, and let $X$ be the Fano variety of lines in $Z$. Let $a\in A^3(X)$ be a $1$--cycle of the form
   \[ a = \displaystyle\sum_{i=1}^r  Y_i\cdot D_i\ \ \ \in A^3(X)\ ,\]
   where $Y_i$ is a Lagrangian constant cycle subvariety and $D_i\in A^1(X)$. Then $a$ is rationally trivial if and only if $a$ is homologically trivial.
   \end{corollary}
   
  \begin{proof} We know from proposition \ref{cc} that each $Y_i$ is in $A^2_{(0)}(X)$. Since $D_i\in A^1(X)=A^1_{(0)}(X)$, it follows that $a$ is in $A^3_{(0)}(X)$.
  But we know \cite{SV} that 
    \[A^3_{(0)}(X)\cap A^3_{hom}(X)=0\ .\]
  \end{proof}

  \begin{corollary}\label{cc3} Let $Z\subset\PP^5(\C)$ be a very general smooth cubic fourfold, and let $X$ be the Fano variety of lines in $Z$.   
  Let $\phi\colon X\dashrightarrow X$ be the degree $16$ rational map defined in \cite{V21}. Let $a\in A^2(X)$ be a $2$--cycle of the form
    \[ a = \phi^\ast(b)-4b\ \ \ \in A^2(X)\ ,\]
    where $b$ is a linear combination of Lagrangian constant cycle subvarieties and intersections of $2$ divisors. Then
    $a$ is rationally trivial if and only if $a$ is homologically trivial.
\end{corollary}

\begin{proof} We know from proposition \ref{cc} that $b$ is in $A^2_{(0)}(X)$. 
Let $V^2_\lambda$ denote the eigenspace 
    \[ V^2_\lambda:=\{ \alpha\in A^2(X)\ \vert\ \phi^\ast(\alpha)=\lambda\cdot \alpha\}\ .\] 
Shen--Vial have proven that there is a decomposition
  \[ A^2_{(0)}(X)=V^2_{31}\oplus V^2_{-14}\oplus V^2_{4}\ \]
    \cite[Theorem 21.9]{SV}. The ``troublesome part'' $A^2_{(0)}(X)\cap A^2_{hom}(X)$ is contained in $V^2_4$ \cite[Lemma 21.12]{SV}. This implies that
   \[ (\phi^\ast-4(\Delta_X)^\ast)A^2_{(0)}(X)=V^2_{31}\oplus V^2_{-14} \]
 injects into cohomology.  
       \end{proof}

\vskip1cm
\begin{nonumberingt} 
Thanks to Jiji and Baba for kindly receiving me in quiet Sayama, where this note started its life. Thanks to Charles Vial for helpful email exchanges.
\end{nonumberingt}

\end{document}